\documentclass[12pt]{amsart}
\makeatletter
\def\@settitle{\begin{center}%
  \baselineskip14\p@\relax
  \bfseries
  \@title
  \end{center}%
}
\makeatother
\usepackage{amssymb} % SRC Specials: DVI [Inverse] Search

\newtheorem{thm}{Theorem}[section]
\newtheorem{cor}[thm]{Corollary}
\newtheorem{lem}[thm]{Lemma}
\newtheorem{prop}[thm]{Proposition}
\newtheorem{exam}[thm]{Example}

\theoremstyle{definition}
\newtheorem{defn}[thm]{Definition}
\newtheorem{rem}[thm]{Remark}

\numberwithin{equation}{section}

\usepackage[colorlinks]{hyperref}
\begin{document}
\title{ On the order of products of elements in finite groups}
\author{M. Amiri}
\address{Universidade Federal de Uberlândia, Faculdade de Matemática, 38408-902 Uberlândia-MG, Brazil}
\email{m.amiri77@gmail.com}
\author{I. Lima$^\ast$}
\address{Universidade de Bras\'ilia, Departamento de Matem\'atica, 70910-900 Bras\'ilia-DF, Brazil}
\email{igor.matematico@gmail.com}
\thanks{The authors were partially supported by FAPDF, Brazil. The third author was partially supported by CAPES, Brazil.}
\author{S. Sousa}
\thanks{$^*$Corresponding author}
\address{Universidade de Bras\'ilia, Departamento de Matem\'atica, 70910-900 Bras\'ilia-DF, Brazil}
\email{sharmenya.andrade@gmail.com}
%\footnotetext{E-mail Address: {\tt m.amiri77@gmail.com;\,  	 igor.matematico@gmail.com;\, sharmenya.andrade@gmail.com } }

\date{}
\maketitle

%------------------------------------------------------------------------------------------------------------------------------
\begin{quote}
{\small \hfill{\rule{13.3cm}{.1mm}\hskip2cm} \textbf{Abstract.} In this work, we introduce the subgroups $D_m(G)$ and $D_{m,n}(G)$ for a group $G$ with finite exponent. We show that both subgroups are characteristic, that $D_{m,n}(G)$ is always nilpotent, and that their nilpotent structure provides a characterization of Frobenius group decompositions. 
Furthermore, we define the $E_{m,n}$-series, which extends this framework to the study of an important class of solvable groups of Fitting height at most $4$. In this sense, we say that a group is $E_{m,n}$-nilpotent if it possesses a finite, strictly increasing $E_{m,n}$-series, and we prove that the maximum length of such a series in an $E_{m,n}$-nilpotent group is 4.

}\\
\vspace{-3mm}\hfill{\rule{13.3cm}{.1mm}\hskip2cm}
\end{quote}
%________________________________________________________________________________________________________________________________________

\section{Introduction}

Proposed in 1902, the Burnside problem asks if a periodic and finitely generated group is necessarily finite. In 1964, Golod and Shafarevich \cite{GOLOD,GOLODS} provided a negative answer to this question. However, this investigation is not finished. The problem has a lot of variants, many of them unsolved to this day. This shows how the study of orders of elements has a central role in group theory. 

 J. Thompson \cite{Thom} showed that a finite group $G$ is non-solvable if and only if there are no elements $x,y,z \in G$ with coprime orders with $xyz=1$. Equivalently, Thompson proved that $G$ is non-solvable if and only if there exist elements $x,y$ with coprime orders such that $o(xy)$ is coprime with $o(x)o(y)$. 
 
Recently, R.M. Guralnick and P.H. Tiep \cite{Tiep} improved this result by replacing the condition of coprime orders by elements with distinct prime power orders, expanding the applicability of the characterization. Baumslag and Wiegold \cite{Baumslag} showed that a finite group $G$ is nilpotent if and only if $o(xy)=o(x)o(y)$ for every $x,y \in G$ with $\gcd(o(x),o(y))=1$. In this same way, A.~Moretó and A.~Sáez \cite{More} established an analogous characterization for nilpotent groups in terms of the prime numbers that divide the order of the products:
\begin{thm}
     Let $G$ be a finite group. The following assertions are equivalent:
     \begin{enumerate}
         \item[(i)] $G$ is nilpotent.
         \item[(ii)] For every $x,y\in G$ with prime power order, with $\gcd(o(x), o(y)) = 1$, it holds that
         $$
         \pi(o(x)o(y))= \pi(o(xy)),
         $$
         where $\pi(m)$ denotes the set of all prime numbers which divides $m$.
     \end{enumerate}
\end{thm}

Our study is also part of this investigation. For $m \ge 1$ we call an element $x \in G$ an $m$-element if $x^m=1$. We denote by $L_m(G)$ the set of all $m$-elements in $G$ and by $D_m(G)$ the set of all $m$-elements $x$ such that $o(xy)\mid m$ for any $y \in L_m(G)$. Analogously, for $n \ge 1$ coprime with $m$, we denote by $D_{m,n}(G)$ the set of all $m$-elements such that $o(xy) \mid n$ for any $y \in L_n(G) \setminus \{1\}$.

It is important to highlight that the set $L_m(G)$ appears in other works, among of them, the centenary work due to Frobenius \cite{Frobenius}, in which they prove that the size of $L_m(G)$ is a multiple of $m$.

In this paper we prove that the sets $D_m(G)$ and $D_{m,n}(G)$ have group structure. Furthermore, they are characteristic subgroups of $G$ and $D_{m,n}(G)$ is always nilpotent, which is a consequence of the following result.
\begin{thm}
  Let $G$ be a finite group with order $mn$ and $\mathrm{gcd}(m,n)=1$. For any $x\in L_n(G) \setminus{\{1\}}$,
$ D_{m,n}(G)\langle x\rangle$ is a Frobenius group with kernel $D_{m,n}(G)$.
In addition, $D_{m,n}(G)\neq 1$ is a $\pi(m)$-Hall subgroup if and only if there exists a $\pi(n)$-Hall subgroup $H$ such that $L_{m}(G)H=D_{m,n}(G)\rtimes H$ is a Frobenius group.
\end{thm}

As an application of the previous result, we have proved a similar result to that done by Moretó and Sáez, which is a substantial inspiration for our study. 
\begin{thm}
   Let $G$ be a finite group of order $mn$, where $m$, $n >1$ and $\gcd(m,n)=1$. For all $x\in L_m(G)$ and $y\in L_n(G)\setminus \{1\}$ with $\gcd(o(x),o(y))=1$ we have $\pi(o(xy))\subseteq \pi(o(y))$ if and only if $G$ is a Frobenius group with complement of order $n$.  
\end{thm}

It is worth noting that our approach is conceptually close to that adopted in a recent work of the first and second named authors on LC-nilpotent groups \cite{Amiri}. In this paper, we associate the structure defined here with those that they introduced in that previous work. In fact, we prove that $D_{m,n}(G)$ is actually a subgroup of $LC(G)$. Besides that, in the case that $m$ is a power of a prime $p$, $LC_p(G)$ is a normal $p$-subgroup of $D_m(G)$.

We investigate how the sets $D_{m,n}(\cdot)$ measure the nilpotency of a group. In fact, we proved that

\begin{thm}
     Let $G$ be a finite group of order $mn$, where $\gcd(m,n)=1$. Then, for any subgroup $M$ of $G$, we have
\[
D_{m,n}\!\bigl(M / Z(M)\bigr) = 1,
\]
if and only if $G$ is a nilpotent group.
\end{thm}

In the final section, we introduce the $E_{m,n}$-series of a finite group $G$, defined recursively from the subgroups $D_{m,n}(\cdot)$ and $D_{n,m}(\cdot)$. This construction alternates the parameters $m$ and $n$ and provides a hierarchy of characteristic subgroups with nilpotent quotients. We show that a group is nilpotent precisely when its $E_{m,n}$-series has length $2$, and more generally, we say that a finite group $G$ is $E_{m,n}$-nilpotent when it has a finite and strictly increasing $E_{m,n}$-series and we show in this case that the length of this $E_{m,n}$-series is always bounded by 4. Also, we provide a characterization of all finite $E_{m,n}$-nilpotent groups.

\section{Properties of $D_m(G)$ and $D_{m,n}(G)$}
Let \(G\) be a finite group and let \(m,n\ge 1\) be coprime integers. In this section we recall the definitions of \(D_m(G)\) and \(D_{m,n}(G)\) and establish their fundamental properties --  behavior under quotients, direct products, and automorphisms -- as well as nilpotency criteria and the structural consequences when these subgroups are nontrivial.

\begin{defn}\label{def2}
   Let $G$ be a periodic group with finite exponent $\exp(G)$ and let $m\geq 1$ and $n\geq 1$ be two positive and coprime integers.
  We denote the set $\{ x\in G \mid x^m=1\}$ by $L_m(G)$, and call every element of $L_m(G)$ a \emph{$m$-element}.
   For $x\in L_m(G)$, we set:
   \begin{enumerate}
       \item[(i)] $D_m(x) = \{ y\in L_m(G) \mid o(xy)\mid m\}$. 
       \item[(ii)] $D_{m,n}(x) =  \{u\in L_n(G) \mid o(xu)\mid n \}$.  
       \item[(iii)] $D_m(G)=\{x\in L_m(G) \mid D_m(x)=L_m(G)\}$.
       \item[(iv)] $D_{m,n}(G)=1$ whenever $\gcd(\exp(G),n)=1$. For $n>1$, we define
       $$
       D_{m,n}(G)=\{x\in L_m(G) \mid L_n(G)\setminus\{1\}\subseteq D_{m,n}(x)\}.
       $$
   \end{enumerate}  
\end{defn}

\begin{lem}\label{Lemma2.1}
   Let $G$ and $H$ be two periodic groups of finite exponent $mn$ where $gcd(m,n)=1$.
Let $x\in L_m(G)$ and $N$ a normal subgroup of $G$.
Then 
\begin{itemize}
    \item [$(i)$] $\frac{D_m(x)N}{N}\subseteq D_m(xN)$ and  $\frac{D_{m,n}(x)N}{N}\subseteq D_{m,n}(xN)$.
    \item [$(ii)$] $D_m(G\times H)=D_m(G)\times D_m(H)$ and $D_{m,n}(G\times H)\subset D_{m,n}(G)\times D_{m,n}(H)$.
    \item [$(iii)$] If $H\leq G$, then $D_m(G)\cap H \leq D_m(H)$ and $D_{m,n}(G)\cap H \leq D_{m,n}(H)$.
\end{itemize}

\end{lem}

\begin{proof}
\begin{itemize}
    \item  [$(i)$] Let $y\in D_m(x)$ and $ z\in {D_{m,n}(x)}$. Then $o(xy)\mid m$ and $o(xz)\mid n$, and so 
   $o(xyN)\mid m$ and  $o(xzN)\mid n$. Hence $yN\in D_m(xN)$ and ${z}N\in D_{m,n}(xN)$.

   \item  [$(ii)$] Let $(a,b) \in D_m(G\times H)$. Taking $x \in L_m(G)$ and $y \in L_m(H)$, we have that $(x,y) \in L_m(G \times H) = D_m((a,b))$. Hence, $o((a,b)(x,y)) \mid m$, which implies that $o(ax)$ and $o(by)$ both divide $m$ and so $x \in D_m(a)$ and $y \in D_m(b)$. Therefore, $(a,b) \in D_m(G) \times D_m(H)$. For the converse, take $(a,b) \in D_m(G) \times D_m(H)$. For any $(x,y) \in L_m(G \times H)$, we have that $x \in L_m(G) = D_m(a)$ and $y \in L_m(H) = D_m(b)$. Hence, $o(ax) \mid m$ and $o(by) \mid m$, which implies that $o((a,b)(x,y)) \mid m$ and so $(x,y) \in D_m((a,b))$. Therefore, $(a,b) \in D_m(G \times H)$.

   \noindent Now, let $(a,b) \in D_{m,n}(G\times H)$. Taking $x \in L_n(G) \setminus \{1\}$ and $y \in L_n(G) \setminus \{1\}$, we have that $(x,y) \in L_n(G\times H) \setminus \{(1,1)\} \subset D_{m,n}(a,b)$. Hence, $o((a,b)(x,y)) \mid n$, which implies that $o(ax)$ and $o(by)$ both divide $n$ and so $x \in D_{m,n}(a)$ and $y \in D_{m,n}(b)$. Therefore, $(a,b) \in D_{m,n}(G) \times D_{m,n}(H)$. The reverse inclusion is not necessarily true (see Example \ref{product} below). 

   \item [$(iii)$] We only write the proof for $D_{m,n}(\cdot)$. The other case is analogous. Let $x \in D_{m,n}(G) \cap H$. Then $o(x) \mid m$, which implies that $x \in L_m(H)$. Now, $L_n(H)\setminus \{1\} \subset L_n(G)\setminus \{1\} \subset D_{m,n}(x)$, because $x \in D_{m,n}(G)$ and $H \leq G$. Thus, $x \in D_{m,n}(H)$. 
\end{itemize}\end{proof}

The assertions of Lemma \ref{Lemma2.1} become clearer in a concrete setting; we therefore present an example that computes \(D_m(\cdot)\) and \(D_{m,n}(\cdot)\) in a specific direct product.

\begin{exam}\label{product}
    If $G=H=S_3$, $m=3$ and $n=2$, then $D_3(G)\cong C_3\cong D_{3,2}(G)$. Also, $D_3(G \times H) \cong C_3 \times C_3$ and $D_{3,2}(G \times H)=1$. This shows why the equality in the last sentence of Lemma \ref{Lemma2.1} (ii) may not occur. Notice that this example also shows that $D_{m,n}(\cdot)$ does not preserve the inclusion relation, since $S_3 \le S_3 \times S_3$, but $D_{m,n}(S_3) \nsubseteq D_{m,n}(S_3 \times S_3)$. In particular, the equality in Lemma \ref{Lemma2.1} (iii), in general, does not occur.
\end{exam}

The previous example shows that the behavior of $D_{m,n}(\cdot)$ under direct products is more subtle than that of $D_m(\cdot)$. The next remark clarifies precisely when equality holds in Lemma \ref{Lemma2.1} (ii).

\begin{rem}
Let $G,H$ be two periodic groups with exponent $mn$, with $\gcd(m,n)=1$. The equality $D_{m,n}(G\times H)=D_{m,n}(G) \times D_{m,n}(H)$ holds if and only if $D_{m,n}(G)=D_{m,n}(H) =1$. Indeed, suppose that $D_{m,n}(G\times H)=D_{m,n}(G) \times D_{m,n}(H)$ and let $(a,b) \in D_{m,n}(G) \times D_{m,n}(H)$. For $y \in L_n(H) \setminus \{1\}$, we have that $(1,y) \in L_{n}(G\times H) \setminus \{(1,1)\}$. Since $(a,b) \in D_{m,n}(G\times H)$, it holds that $o((a,b)(1,y)) \mid n$, which implies that $o(a) \mid n$. Since $\gcd(m,n)=1$, this implies that $a=1$. With a similar argument, we can show that $b=1$.   
\end{rem}
After examining how the sets $D_m(\cdot)$ and $D_{m,n}(\cdot)$ behave with respect to quotients and direct products, we investigate how they behave with respect to automorphisms.

\begin{lem}\label{ces} It holds that
\begin{itemize}
    \item [$(i)$] For any $\sigma\in Aut(G)$, we have $D_m(G)^{\sigma}=D_m(G)$ and $D_{m,n}(G)^{\sigma}=D_{m,n}(G)$.

    \item [$(ii)$] $\langle D_m(G)\rangle=D_m(G)$ and  $\langle D_{m,n}(G)\rangle=D_{m,n}(G)$.

    \item [$(iii)$] If $G$ is locally finite, then $D_{m,n}(G)\subseteq D_m(G).$
\end{itemize}

%(ii) The exponents of the subgroup generated by $D_m(G)$ and $D_{m,n}(G)$ divide $m$.
\end{lem}
\begin{proof}   
\begin{itemize}
    \item [$(i)$] Let $x\in D_m(G)$ and let $\sigma\in Aut(G)$.
Let $y$ be an $m$-element of $G$.
Then there exists an $m$-element $z\in G$ such that $\sigma(z)=y$.
Then $$o(\sigma(x)y)=o(\sigma(xz))=o(xz)\mid m.$$ 
This proves that $D_m(G)^\sigma \subset D_m(G)$. The converse is analogous.

Let $x\in D_{m,n}(G)$ and let $\sigma\in Aut(G)$.
Let $y \in L_n(G) \setminus \{1\}$. Then there exists $z \in L_n(G) \setminus \{1\}$ such that $\sigma(z)=y$.
Then $$o(\sigma(x)y)=o(\sigma(xz))=o(xz)\mid n.$$ 
This proves that $D_{m,n}(G)^\sigma \subset D_{m,n}(G)$. The converse is again analogous.

\item [$(ii)$] Let $x,y\in D_m(G)$, and let $z\in L_m(G)$. Notice that 
$$
o(xz^{-1}) =o((xz^{-1})^{-1})=o(z^{-1}(xz)^{-1}z)=o(x^{-1}z),
$$which implies that $x^{-1} \in D_m(G)$. Since $y^{-1}\in D_m(G)$, we have $y^{-1}z\in L_m(G)$, so
    $o(x(y^{-1}z))\mid m$, and hence $xy^{-1}\in D_m(G)$.
    It follows that $D_m(G)\leq G$.

    Now, we prove $D_{m,n}(G)\le G$.
    Let $x\in D_{m,n}(G)$, and $z\in L_n(G)\setminus\{1\}$. Repeating the previous argument, we have that $x^{-1}\in D_{m,n}(G)$.
    
Let $x,y\in D_{m,n}(G)$.
We claim that
\[
xy\in L_m(G).
\]

If  $xy=yx$ then 
\[
o(xy)\mid \operatorname{lcm}(o(x),o(y))\mid m,
\]
so the claim holds.

Thus, assume that $xy\ne yx$.
Then there exist elements
\[
u\in L_m(G) \quad \text{and} \quad v\in L_n(G)
\]
such that
\[
xy=uv=vu.
\]

Suppose, for a contradiction, that $v\neq 1$.
Since $x,y\in D_{m,n}(G)$, we have
\[
o(yv^{-1})\mid n.
\]

If $yv^{-1}=1$, then $y=v\in L_m(G)\cap L_n(G)=\{1\}$.
Hence $y=1$, which implies $xy=yx$, a contradiction.

Therefore $yv^{-1}\neq 1$.
It follows that
\[
o(u)=o\!\left(x(yv^{-1})\right)\mid n,
\]
and since $u\in L_m(G)$ with $\gcd(m,n)=1$, we conclude that $u=1$.

Consequently,
\[
xy=v\in L_n(G).
\]
Then
\[
o\bigl(x^{-1}(xy)\bigr)=o(y)\mid n.
\]
Since $y\in L_m(G)$, we also have $o(y)\mid m$, and again by
$\gcd(m,n)=1$ it follows that $y=1$.
This implies $xy=yx$, a contradiction.
Hence $v=1$, and therefore $xy\in L_m(G)$, as claimed.

Let  
    $z\in L_n(G)\setminus \{1\}$.
    Since $y^{-1}\in D_{m,n}(G)$, we have $y^{-1}z\in L_n(G)$.
    If $y^{-1}z=1$, then $y=z\in L_n(G)\cap L_m(G)=1$, which is a contradiction.
    So $y^{-1}z\in L_n(G)\setminus \{1\}.$
    Since $x\in D_{m,n}(G)$, we have 
    $o(x(y^{-1}z))\mid n$, and hence $xy^{-1}\in D_{m,n}(G)$.
    It follows that $D_{m,n}(G)\leq G$. 

\item [$(iii)$] By the previous items, $D_{m,n}(G) \trianglelefteq G$. Let $x\in  D_{m,n}(G)$ and let $y\in L_{m}(G)$. Since $G$ is locally finite, 
$$\exp( D_{m,n}(G)\langle y\rangle)=\mathrm{lcm}(\exp(D_{m,n}(G)),\exp(\langle y \rangle))\mid m,$$ because $\exp( D_{m,n}(G)) \mid m$.
It follows that $o(xy)\mid m$, and so $x\in D_m(G)$.
Hence $D_{m,n}(G)\subseteq D_m(G)$.
\end{itemize}
\end{proof}

The following remark shows that $D_{m,n}(G)$ cannot be too large: despite its stability under automorphisms, it is always a proper subgroup whenever $n>1$.

\begin{rem}\label{proper}
If $n>1$ and $(m,n)=1$, $D_{m,n}(G)$ is always a proper subgroup of $G$. Indeed, suppose for a contradiction that $D_{m,n}(G) = G$. Once we know that $D_{m,n}(G) \subseteq L_m(G)$, we would have $G = L_m(G)$. Since $\exp(G)=mn$, it follows that $n =1$, which is a contradiction. 
\end{rem}

The next observation highlights how the nature of
$m$ influences the internal structure of $D_m(\cdot)$ and $D_{m,n}(\cdot)$, and in particular when these subgroups may fail to be nilpotent.

\begin{rem}
If $m$ is a power of a prime $p$, then every element of $L_m(G)$ has $p$-power order. Consequently $D_m(G)$ and $D_{m,n}(G)$ consist of $p$-elements and, when finite, are $p$-groups. But if $m$ is not a power of a prime, then 
$D_m(G)$ is not necessarily a nilpotent group; for example,
if $G=S\times C_p$ where $S$ is a non-abelian finite simple group and  $p$ is a prime number such that $p\nmid |S|=m$, then $D_{m}(G)=S$ which is not a solvable group. Finally, if $G$ is finite with $|G|=mn$, $\gcd(m,n)=1$, and $n>1$ the subgroup $D_{m,n}(G)$ is nilpotent (see Theorem \ref{nil} and Corollary \ref{nil2}).
\end{rem}

Frow now on, we deal specifically with finite groups with order $mn$ instead of the more general approach with finite exponent $mn$. 

We now arrive at a key result describing the precise group-theoretic structure carried by $D_{m,n}(G)$. In fact, these subgroups naturally arise as kernels of Frobenius group decompositions.

\begin{thm}\label{nil}
  Let $G$ be a finite group with order $mn$ and $\mathrm{gcd}(m,n)=1$. For any $x\in L_n(G)\setminus\{1\}$,
$ D_{m,n}(G)\langle x\rangle$ is a Frobenius group with kernel $D_{m,n}(G)$.
In addition, $D_{m,n}(G)\neq 1$ is a $\pi(m)$-Hall subgroup if and only if there exists a $\pi(n)$-Hall subgroup $H$ such that $L_{m}(G)H=D_{m,n}(G)\rtimes H$ is a Frobenius group.
\end{thm}
\begin{proof}
 Suppose that $D_{m,n}(G)\langle x\rangle$ is not a Frobenius group.
Then $C_{\langle x\rangle}(y)\neq \{1\}$ for some $y \in D_{m,n}(G) \setminus \{1\}$. Let $x^i\in C_{\langle x\rangle}(y)\setminus\{1\}$.
Then $o(yx^i)=o(y)o(x^i)\nmid n$, which is a contradiction, since $y \in D_{m,n}(G)$.
Hence $D_{m,n}(G)\langle x\rangle$ is a Frobenius group.

($\Longrightarrow$) 
If $D_{m,n}(G)\neq 1$ is a $\pi(m)$-Hall subgroup, then by Schur–Zassenhaus there exists a $\pi(n)$-Hall subgroup $H$ such that $G=D_{m,n}(G)\rtimes H$.
By the first part, for any $x\in H$, we have $D_{m,n}(G)\langle x\rangle$ is a Frobenius group; then the action of $H$ on $D_{m,n}(G)$ is fixed-point-free, and so $D_{m,n}(G)\rtimes H$ is a Frobenius group.

($\Longleftarrow$) 
If there exists a $\pi(n)$-Hall subgroup $H$ such that $L_m(G)H=D_{m,n}(G)\rtimes H$ is a Frobenius group, then $L_m(G)=D_{m,n}(G)$ is a $\pi(m)$-Hall subgroup of $G$.
\end{proof}
An immediate consequence of the previous theorem is the following structural description of $D_{m,n}(G)$.
\begin{cor}\label{nil2}
    Let $G$ be a finite group of order  $mn$ where $gcd(m,n)=1$ with $n>1$. 
    Then $D_{m,n}(G)$ is a characteristic nilpotent subgroup of $G$.
\end{cor}
\begin{proof}
It follows from Lemma \ref{ces} that $D_{m,n}(G)$ (and $D_m(G)$) are characteristic subgroups of $G$. The nilpotency of $D_{m,n}(G)$ is due to the fact that $D_{m,n}(G)$ is the kernel of a Frobenius group. 
\end{proof}

To proceed, we recall two useful facts about orders of products in semidirect decompositions, especially in the Frobenius setting.

\begin{lem}\label{2.7}
    Let $G = K \rtimes H$ be a group with $K \trianglelefteq G$ and $K \cap H =1$. 
\begin{itemize}
    \item [$(a)$] For every $x \in K$ and $y \in H$ we have that $o(y) \mid o(xy)$. 
    \item [$(b)$] If $G$ is a Frobenius group with kernel $K$ and complement $H$, $o(xy) = o(y)$ for every $x \in K$ and $y \in H\setminus \{1\}$. 
\end{itemize}
\begin{proof}
    \begin{itemize}
        \item [$(a)$] Let $o(xy) = m$ and $o(y) = n$. We have that
    \begin{eqnarray*}
    1 &=& (yx)^m = yxyx \cdots yx \\
    &=& yxy^{-1}yyxy^{-2}y^2yx \cdots y^{m-1}yxy^{-m}y^m \\
    &=& x^y x^{y^2} \cdots x^{y^m} y^m \implies z:= x^y x^{y^2} \cdots x^{y^m} = y^{-m} \in H.
    \end{eqnarray*}
However, since $x \in K \trianglelefteq G$, $z \in K$. In this case, $z \in K \cap H =1$. Hence, $y^m =1$, which shows that $o(y) \mid m = o(xy)$.
\item [$(b)$] Consider $n = o(y)$. Then 
  \begin{eqnarray*}
    (xy)^n &=& xx^{y^{-1}}x^{y^{-2}}\cdots x^{y^{-(n-1)}}y^n \\
           &=& xx^{y^{-1}}x^{y^{-2}}\cdots x^{y^{-(n-1)}} \in K,
    \end{eqnarray*}
since $K \trianglelefteq G$. Thus, $(xy)^n \in K$, that is, $o((xy)^n) \mid |K|$. On the other hand, since $H \cap K = 1$ and $y \neq 1$, then $xy \not\in K$, because if $xy \in K$, then $x^{-1}xy = y \in K$, which would imply $y \in K \cap H = 1$. Hence, as
$$K = G \setminus \bigg(\bigcup_{g \in G} H^g \bigg) \implies xy \in \bigcup_{g \in G} H^g.$$
So, there exists $g \in G$ such that $xy \in H^g$, that is, $o(xy) \mid |H^g|$. However, $o((xy)^n) \mid o(xy)$, since $(xy)^n \in \langle xy \rangle$, thus,
$$o((xy)^n) \mid |H| \implies o((xy)^n) \mid \gcd(|H|, |K|) = 1 \implies o((xy)^n) = 1 $$
$$ \implies (xy)^n = 1 \implies o(xy) \mid n = o(y).$$
However, from item $(a)$, we have that $o(y) \mid o(xy)$. Therefore, $o(xy) = o(y)$.
    \end{itemize}
\end{proof}
\end{lem}

We now combine the preceding lemma with the structure of $D_{m,n}(G)$ to obtain a characterization reminiscent of the result of Moretó and Saéz \cite{More}.

\begin{thm}\label{Fro1}
Let $G$ be a finite group of order $mn$, where $m$, $n >1$ and $\gcd(m,n)=1$. For all $x\in L_m(G)$ and $y\in L_n(G)\setminus \{1\}$ with $\gcd(o(x),o(y))=1$ we have $\pi(o(xy))\subseteq \pi(o(y))$ if and only if $G$ is a Frobenius group with complement of order $n$.  
\end{thm}

\begin{proof} 

($\Longrightarrow$) 
    Suppose that for any $m$-element $x$ and $n$-element $y\neq 1$ of $G$, we have $\pi(o(xy))\subseteq \pi(o(y))$.
Let $x\in L_m(G)$, and let $y\in L_n(G)\setminus\{1\}$.
We have $\pi(o(xy))\subseteq \pi(o(y))\subseteq \pi(n)$.
Hence $o(xy)\mid n$. It follows that $D_{m,n}(x)=L_{n}(G)\setminus\{1\}$, and so 
$D_{m,n}(G)=L_m(G)$. 
Thus $D_{m,n}(G)$ is the set of all $m$-elements.

From Corollary \ref{nil2}, $D_{m,n}(G)$ is a characteristic nilpotent subgroup of $G$; since it contains every $p$-Sylow subgroup of $G$ with $p\mid m$, it follows that $|D_{m,n}(G)|=m$, i.e.\ $D_{m,n}(G)$ is a $\pi(m)$-Hall subgroup of $G$. By Schur--Zassenhaus there exists a complement $H$ with $|H|=n$, and from Theorem \ref{nil} the action of $H$ on $D_{m,n}(G)$ is fixed-point-free on nontrivial elements; hence $G=D_{m,n}(G)\rtimes H$ is a Frobenius group with complement of order $n$.

($\Longleftarrow$)  
    If $G=K\rtimes H$ is a Frobenius group with kernel $K$ of order $m$ and complement $H$ of order $n$, then if $x\in L_m(G)$ it holds that $x\in K$. Furthermore, since $y\in L_n(G)\setminus \{1\}$, $y \notin K$. Thus, $y \in H^g$ for some $g \in G$. Since $H^g$ also can be viewed as a Frobenius complement, it follows from Lemma \ref{2.7} that  $o(xy)=o(y)$, and so $\pi(o(xy))\subseteq \pi(o(y))$.
\end{proof}

Having characterized when $G$ itself is a Frobenius group, we next study the consequences of factoring out $D_{m,n}(G)$ and the restrictions this imposes on Sylow subgroups.

\begin{prop}\label{factor}
Let $G$ be a finite group with order $mn$ where $\gcd(m,n)=1$. 
\begin{itemize}
    \item[(i)] Then $D_{m,n}\bigl(G/D_{m,n}(G)\bigr)=1$.

\item[(ii)] Suppose $D_{m,n}(G)\neq 1$, and let $P\in \mathrm{Syl}_p(G)$ where $p\mid n$.
Then $P$ is cyclic or generalized quaternion. Furthermore, suppose $p=2\mid n$ and either $P\ncong Q_{2^t}$ for some $t\geq 3$ or $3\nmid |G|$. Then $G=KH$ where $K$ is a Hall $\pi(m)$-subgroup, and $H$ a Hall $\pi(n)$-subgroup of $G$.
    
\end{itemize}
\end{prop}

\begin{proof} \begin{enumerate}

    \item [(i)] Set $N=D_{m,n}(G)$. If $N=1$, the result follows immediately. So, without lost of generality we can suppose that $N\neq 1$. By contradiction, assume that $D_{m,n}(G/N)\neq 1$ and let $zN\in D_{m,n}(G/N)\setminus\{N\}$. Since $\gcd(m,n)=1$, there exist $r,s \in \mathbb{Z}$ such that $mr+ns=1$. Define $a=z^{mr}$ and $b=z^{ns}$. Thus, $z=ab$, $a \in L_n(G)$ and $b \in L_m(G)$. Since $zN \in D_{m,n}(G/N)$, it follows that $o(zN)\mid m$ and so $z^m \in N$. Hence, $a= (z^m)^r \in N \subset L_m(G)$. Therefore, $a \in L_m(G) \cap L_n(G)=1$ and so $z=b \in L_m(G)$. 

Let $y\in L_n(G)\setminus\{1\}$. Since $|N|\mid m$
and $\gcd(m,n)=1$, we have that $yN\neq N$. Furthermore, $yN \in L_n(G/N) \setminus \{N\}$. This implies that $o(zyN)\mid n$.

Consider $g=zy$ and set $H:=N\langle g\rangle$. Notice that $H/N=\langle gN\rangle$. Hence, $|H/N|=o(gN)\mid n$.
Once that $|N|\mid m$ and $\gcd(m,n)=1$,
we conclude that $N$ is a $\pi(m)$-Hall subgroup of $H$.
By Schur--Zassenhaus Theorem, there exists a complement $C$ of $N$ in $H$ such that $H=N\rtimes C$. In particular,
$C\cong H/N$ is cyclic with order dividing $n$. From Theorem \ref{nil}, we obtain that $H$ is a Frobenius group with kernel $N$ and complement $C$. In particular, $C_N(c) =1$ for any $c \in C \setminus \{1\}$. Since $g \in H = NC$, there exist $n \in N$ and $c \in C$ such that $g=nc$. Thus, $gN=cN$. Since $z \notin N$, it follows that $c \neq 1$. 

Define $\phi : N \to N$ by $\phi(t) = [t,c^{-1}] = t^{-1}ctc^{-1}$. Observe that $\ker \phi = C_N(c)= 1$ and so $\phi$ is bijective. For this reason, there is $t \in N$ such that $n = t^{-1}ctc^{-1}$. Remember that $g= nc= t^{-1}ct = c^t$. Therefore, $o(g)=o(c^t)=o(c) \mid n$. So, $o(zy) \mid n$. Since $z \in L_m(G)$ and $y \in L_n(G)\setminus \{1\}$ was arbitrarily chosen, it follows that $z \in D_{m,n}(G)= N$ and hence $zN =N$, a contradiction. 

This contradiction shows that $D_{m,n}(G/N) =1$, concluding the proof.

\item [(ii)]  Let $P\in \mathrm{Syl}_p(G)$ where $p\mid n$.
Then $D_{m,n}(G)P$ is a Frobenius group, and so
$P$ is cyclic or generalized quaternion. Now, suppose $p=2 \mid n$. If $P\ncong Q_{2^t}$ for some $t\geq 3$, then $P$ is cyclic and the result follows by Burnside's $p$-complement theorem. If $3 \nmid |G|$, from the main theorem of \cite{quat} we obtain $G=K\rtimes H$ where $H$ is a $\pi(n)$-Hall subgroup.
\end{enumerate} \end{proof} 
Remember that a group $G$ is called a 2-Frobenius group if there exist normal subgroups $L,K \trianglelefteq G$ such that $L$ is a Frobenius group with kernel $K$ and $G/K$ is a Frobenius group with kernel $L/K$.

\begin{thm}\label{frob}
    Let $G$ be a finite group. If $D_{n,m}\bigl(G/D_{m,n}(G)\bigr)=U/D_{m,n}(G)$ and $D_{m,n}\bigl(G/U\bigr)=V/U$, then 
     \begin{itemize}
         \item  [(a)] $D_{m,n}\bigl(G/V\bigr)=1.$
        \item [(b)] Suppose $D_{m,n}(G) \neq 1$. If $U\neq V$, then $V$ is a $2$-Frobenius group, and if $U\neq D_{m,n}(G)$, then $U$ is a Frobenius group. 
     \end{itemize}
\end{thm}
\begin{proof}
\begin{itemize}
    \item [(a)] By applying Proposition \ref{factor} (i) to $G/U$, we have
$$
 1= D_{m,n}\left(\dfrac{G/U}{D_{m,n}(G/U)}\right) = D_{m,n}\left(\dfrac{G/U}{V/U}\right) \implies D_{m,n}\left(\dfrac{G}{V}\right) = 1,
$$
as desired. 
\item [(b)] If $U\neq D_{m,n}(G)$, then $D_{n,m}\bigl(G/D_{m,n}(G)\bigr) = U/D_{m,n}(G) \neq 1$. 
For $x \in D_{m,n}(G)$, we have $\pi(o(x)) \subset \pi(m)$. On the other hand, 
$U/D_{m,n}(G)$ is a $\pi(n)$-group. Once $(m,n) = 1$, if $p \mid |D_{m,n}(G)|$ is a prime number, then $p \nmid |U|/|D_{m,n}(G)|$. Therefore, $|D_{m,n}(G)|$ and $|U|/|D_{m,n}(G)|$ are coprime. Hence, $D_{m,n}(G)$ is a Hall subgroup of $U$. By Schur–Zassenhaus Theorem, there exists a complement $H \leq U$ such that $U = D_{m,n}(G) \rtimes H$. 

Let $1 \neq h \in H$ and $1 \neq a \in D_{m,n}(G)$ with $a \in C_{D_{m,n}(G)}(h)$.
Since $a \in D_{m,n}(G)$, $o(a) \mid m$. Once $h$ is an $n$-element, it holds that $o(ah) \mid n$. As $a$ and $h$ commute, $o(ah) = \mathrm{lcm}(o(a), o(h)) = o(a)o(h)$. Consequently, there exists $q \in \pi(m)$ such that $q \mid o(ah) \mid n$, which is a contradiction. So, $a = 1$, which shows that $C_{D_{m,n}(G)}(h) = 1$ for every $h \in H \setminus \{1\}$. Hence, $U$ is a Frobenius group. 

If $U \neq V$, then $D_{m,n}\bigl(G/U\bigr) = V/U \neq 1$. 
Let $A = D_{m,n}(G)$ and consider the quotient $\overline{G} = G/A$, 
and denote $\overline{U} = U/A$ and $\overline{V} = V/A$. 
By construction, $\overline{U}$ is a $\pi(n)$-group and
$\overline{V}/\overline{U} = V/U$ is a $\pi(m)$-group. 
Since $(m,n)=1$, it follows that $|\overline{U}|$ and $|\overline{V}|/|\overline{U}|$ are coprime. 
Therefore, $\overline{U}$ is a Hall subgroup of $\overline{V}$. 
By Schur–Zassenhaus, there exists a complement $K \leq \overline{V}$ 
such that $\overline{V} = \overline{U} \rtimes K$. 

Let $1 \neq k \in K$ and $1 \neq u \in \overline{U}$ with $u \in C_{\overline{U}}(k)$. 
Since $u \in \overline{U}$, we have $o(u) \mid n$. 
On the other hand, $k$ is an $m$-element, so $o(k) \mid m$. 
Since $(m,n)=1$, we conclude that $\gcd(o(u),o(k))=1$. 
Furthermore, as $u$ and $k$ commute, 
$$o(uk) = \operatorname{lcm}(o(u),o(k)) = o(u)o(k).$$

However, $o(uk) \mid m$. 
But $o(uk)=o(u)o(k)$ has prime divisors both in $\pi(n)$ and in $\pi(m)$, 
which contradicts the coprimality between $m$ and $n$. Consequently, $u=1$, and $C_{\overline{U}}(k)=1$ for all $k \in K \setminus \{1\}$. 

This shows that $\overline{V} = \overline{U} \rtimes K$ is a Frobenius group with kernel $\overline{U} = U/A$. Since $U$ is a Frobenius group with kernel $A$, we conclude that $V$ is a 2-Frobenius group. 
\end{itemize}
\end{proof}

The next result links our construction to the theory of LC-nilpotent groups, showing that these two frameworks interact in a predictable way when 
$m$ is a prime power. Indeed, for a periodic group $G$ and a prime number $p$ denote by $\operatorname{LCM}_p(G)$ the set of all $p$-elements $x \in G$ such that $o(x^ny) \mid \operatorname{lcm}(o(x^n),o(y))$ for any $p$-element $y \in G$. Also, write $\operatorname{LC}_p(G) = \langle \operatorname{LCM}_p(G)\rangle$ (more details on this subject can be found in \cite{Amiri}). We have that 

\begin{prop} Let $G$ be a finite group and let $m$ be a power of a prime $p$. Then, $\operatorname{LC}_p(G) \le O_p(D_m(G))$, where $O_p(D_m(G)))$ denotes the $p$-core of $D_m(G)$.  
\end{prop}

\begin{proof} Let $x \in LC_p(G)$. By \cite[Proposition 2.4]{Amiri}, $x \in \operatorname{LC}(G)$. Since $m$ is a power of $p$, $x \in L_m(G)$. For $y \in L_m(G)$, it holds that
$$
o(xy)\mid \operatorname{lcm}(o(x),o(y)) \mid m.
$$
Thus, $y \in D_m(x)$ and hence $x \in D_m(G)$. So, $LC_p(G)\le D_m(G)$. Finally, since $LC_p(G) \trianglelefteq G$ (see \cite[Lemma 2.3]{Amiri}), it follows that $\operatorname{LC}_p(G)\leq O_p(D_{m}(G))$. 
\end{proof}

We end the section with a nilpotency criterion that depends only on the behavior of $D_{m,n}(\cdot)$ inside subgroups of $G$.
\begin{thm}\label{Min1}
Let $G$ be a finite group. If 
\[
D_{m,n}\left(\frac{M}{Z(M)}\right)=1
\]
for every subgroup $M\leq G$ and every pair of integers
$m,n>1$ satisfying
\[
mn=|G|
\quad\text{and}\quad
\gcd(m,n)=1,
\] then $G$ is nilpotent. 
\end{thm}

\begin{proof}
Assume, for a contradiction, that $G$ is not
nilpotent. Since $G$ is finite, it contains
a subgroup $S$ which is minimal with respect
to being non-nilpotent. Thus $S$ is a Schmidt
group. So, by \cite[Lemma 1]{Schmidt}, there exist unequal prime numbers $p,q$ and $P,Q \leq S$ such that $P$ is a normal Sylow $p$-subgroup of $S$, $Q$ is a cyclic Sylow $q$-subgroup of $S$ and
$$
S = P\rtimes Q.
$$
Moreover, $P'=\Phi(P)$ and setting $Q:=\langle a\rangle$, it holds that 
\[
Z(S)=P'\times\langle a^q\rangle,
\]
and $V:=P/P'$ is a non-central chief factor of $S$. After some computations, we can show that 
$$
\dfrac{S}{Z(S)} \cong V \rtimes C_q,
$$
where $C_q$ denotes the cyclic group of order $q$. Notice that $S$ acts on $V$ via conjugation. Since $V$ is a non-central chief factor of $S$, such action is non-trivial, i.e, there exist $s \in S$ and $x \in P$ such that 
$$
s^{-1}xs P' \neq xP'. 
$$
As $S=PQ$, we can write $s=uv$, $u \in P$ and $v \in Q$. Since $P/P'$ is abelian, $v^{-1}xvP'=s^{-1}xsP' \neq xP'$. Once $Q$ is cyclic, we conclude that $a^{-1}xaP'\neq xP'$. Thus, the action of $Q$ on $V$ is non-trivial. Consider $\phi: V \to V$ the homomorphism induced by this action. 

We claim that $\phi$ is irreducible. Indeed, suppose that there is $1 < W<V$ such that $\phi(W)=W$. By Correspondence Theorem, there exist $P' \triangleleft K < P$ such that $W=K/P'$. Once that $V$ is abelian, $K \triangleleft P$. So, $P \leq N_S(K)$. We will show that $K \trianglelefteq S$. Furthermore, since $Q$ is cyclic and $W$ is $\phi$-invariant it is true that $Q \leq N_S(K)$. As $S=PQ$, we conclude that $K \trianglelefteq S$. We proved the existence of a normal subgroup $K \trianglelefteq S$ such that $P' < K < P$, which is a contradiction because $P/P'$ is a chief factor of $S$. Thus, either $W=1$ or $W=V$ and $\phi$ is irreducible. 

Now, consider $C_q \cong Q/{\langle a^q\rangle}$ and remember that $\langle a^q \rangle \leq Z(S)$. Once that $S/Z(S) \cong V \rtimes C_q$, $C_q$ acts non-trivially on $V$ via conjugation. This action is induced by $\phi$ and hence is irreducible. Consider $W=C_V(C_q)$ which is invariant by the action of $C_q$ on $V$. So, either $W=1$ or $W=V$. If $W=V$, the action of $C_q$ on $V$ would be trivial, a contradiction. Then, $W=1$. Since $C_q$ is cyclic, it follows that $C_V(c)=1$ for any $1\neq c \in C_q$. Thus, $F:= S/Z(S)$ is a Frobenius group whose kernel is $V$ and complement is $C_q$. 

Write $|G| = p^tn$ where $p\nmid n$ and set $m=p^t$. Since $q \mid |G|$, $q \mid n$ and so $m,n>1$ and $\gcd(m,n)=1$. Thus, $D_{m,n}(F)=1$. On other hand, let $x \in V$. Since $V$ is a $p$-group, $o(x) \mid m$. As $F\cong V \rtimes C_q$, for $y \in L_n(F) \setminus \{1\}$ we have that $o(y)=q$ and so $xyV=yV$, which implies that $xy \notin V$. 

We claim that every element in $F \setminus V$ has order $q$.  Let $u \in F\setminus V$ and write $u=vc$, with $v \in V$ and $c \in C_q$. Since $C_V(c)=1$, the map $\varphi: V \to V$ defined by $\varphi(z) = z^{-1}czc^{-1}$ is bijective. So, there is $z \in V$ such that 
$$
v = \varphi(z) = z^{-1}czc^{-1} \implies u = vc = c^z.
$$
Thus, $o(u) = o(c^z)=o(c)=q$. In particular, $o(xy) = q \mid n$ and so $x \in D_{m,n}(F)$. This proves that $1 \neq V \leq D_{m,n}(F) = 1$, a contradiction. 

This final contradiction shows that $G$ is nilpotent, concluding the proof. 
\end{proof}

If $G$ is a finite nilpotent group and $x,y \in G$ are elements of coprime order, then $o(xy) = o(x)o(y)$. By this fact, it follows that 

\begin{lem}\label{2.16}
    Let $G$ be a finite nilpotent group, with order $mn$ such that $\gcd(m,n)=1$. Then $D_{m,n}(G) = 1$.
\end{lem}

Combining both the results, we obtain that
\begin{cor}
    Let $G$ be a finite group. Then $G$ is nilpotent if and
only if
\[
D_{m,n}\left(\frac{M}{Z(M)}\right)=1
\]
for every subgroup $M\leq G$ and every pair of integers
$m,n>1$ satisfying
\[
mn=|G|
\quad\text{and}\quad
\gcd(m,n)=1.
\]
\end{cor}

We end this section with an analogous result, this one for Frobenius groups.
\begin{thm}\label{Min2}
    Let $G$ be a finite group of order $mn$, where $\gcd(m,n)=1$ and $n>1$, such that $D_{m,n}(G)\neq 1$. If for every proper non-nilpotent subgroup $M$ of $G$ we have
\[
D_{m,n}\!\bigl(M / Z(M)\bigr) = 1,
\]
then $G$ is a Frobenius group with kernel $D_{m,n}(G)$.
\end{thm}
\begin{proof}

Since $D_{m,n}(G)\trianglelefteq G$ is solvable, there exists an elementary abelian $p$-group $N\leq D_{m,n}(G)$ which is a minimal normal subgroup of $G$.

Let $M$ be a maximal subgroup of $G$. We have $N\cap M=1$. Indeed, let $K:=N\cap M\triangleleft M$ and suppose $K\neq 1$. Since $N$ is minimal normal in $G$, $KG=N$. By maximality of $M$, either $NM=G$ or $N\subseteq M$. If $NM=G$, then $N=K^G=K^N=K$, because $N$ is abelian. Consequently $N\leq M$, which is a contradiction. Thus $NM\neq G$ and hence $N\leq M$. Let $x\in N$. Since $N\leq D_{m,n}(G)$ we have $o(x)\mid m$ and $o(xy)\mid n$ for any $y\in L_n(M)\setminus\{1\}$. Thus $x\in D_{m,n}(M)$ and we have proved $N\leq D_{m,n}(M)=1$, a contradiction.

Therefore $N\cap M=1$, which implies $G=N\rtimes M$. By the same argument, we may show that $D_{m,n}(G)\cap M=1$. We claim that $D_{m,n}(G)=N$. In fact, if $h\in D_{m,n}(G)\setminus N$, we can write $h=xy$, with $x\in N$ and $1\neq y\in M$. Thus $y=x^{-1}h\in D_{m,n}(G)$, which is absurd. So $D_{m,n}(G)=N$ and $D_{m,n}(G)$ is elementary abelian of prime-power order.

Let $K\leq M$ be a Hall $p'$-subgroup of $M$. By the same argument, we can show that $D_{m,n}(G)\cap R=1$ for any maximal subgroup $R$ of $G$. Then $D_{m,n}(G)$ is a maximal subgroup of $G$ and we have $G=D_{m,n}(G)K$, and so $M=K$ and $M$ is a Hall $p'$-subgroup of $G$.

Let $y\in M\setminus D_{m,n}(G)$ such that $o(yD_{m,n}(G))$ is a power of prime number $q\mid n$. If $D_{m,n}(G)\langle y\rangle\neq G$, then $D_{m,n}\big(D_{m,n}(G)\langle y\rangle\big)\neq 1$, which is a contradiction.  So   $D_{m,n}(G)\langle y\rangle= G$.
 By Theorem \ref{nil}, $G$ is a Frobenius group with kernel $D_{m,n}(G)$.
\end{proof}

\section{$E_{m,n}$-series}

\begin{defn}
Let $G$ be a finite group and let $m,n\in\mathbb{N}$ with $\gcd(m,n)=1$ be fixed parameters. We define an \emph{$E_{m,n}$-series} of $G$ as the sequence (which can be infinite)
$$
1 =: E_{0}(G) \leq E_{1}(G) \leq E_{2}(G) \leq \cdots
$$
recursively defined as follows: for $i \geq 0$, we set
    $$
    E_{2i+1}(G)/E_{2i}(G) \;:=\;
    \begin{cases}
       G/E_{2i}(G), & \text{if $G/E_{2i}(G)$ is nilpotent}, \\[6pt]
       D_{m,n}\!\big(G/E_{2i}(G)\big), & \text{otherwise}.
    \end{cases}
    $$
and
    $$
    E_{2i+2}(G)/E_{2i+1}(G) \;:=\;
    \begin{cases}
       G/E_{2i+1}(G), & \text{if $G/E_{2i+1}(G)$ is nilpotent}, \\[6pt]
       D_{n,m}\!\big(G/E_{2i+1}(G)\big), & \text{otherwise}.
    \end{cases}
    $$

We say that $G$ is $E_{m,n}$-nilpotent if there exists a positive integer $k$ such that
$$
1 = E_{0}(G) < E_{1}(G) < E_{2}(G) < \cdots < E_{k}(G) = G.
$$
In this case, we say that the $E_{m,n}$-series has length $k+1$.
\end{defn}

Each step extracts a characteristic subgroup determined by the operators  $D_{m,n}(\cdot)$ and $D_{n,m}(\cdot)$, thereby providing a way to measure how far $G$ is from being nilpotent.

\begin{rem} The alternation between $D_{m,n}(.)$ and $D_{n,m}(.)$ is essential. By Proposition \ref{factor}(i), if we applied only $D_{m,n}(.)$ repeatedly and $G/D_{m,n}(G)$
were not nilpotent, the sequence would result in the repeated series $1 \le D_{m,n}(G) \le D_{m,n}(G),\dots$ and so $E_{j+1}/E_j=1$ for any $j>1$. By alternating the two operators, we allow the series to grow beyond $D_{m,n}(G)$.
\end{rem}
In the following, we show some examples of $E_{m,n}$-series in specific groups
\begin{exam}
Consider $G = S_3$, $m = 3$ and $n = 2$. We have that
\begin{itemize}
    \item $E_{0}(S_3) = 1$;
    \item $E_{1}(S_3) = D_{3,2}(S_3) = A_3$;
    \item $\dfrac{E_{2}(S_3)}{A_3} = D_{2,3}\!\left(\dfrac{S_3}{A_3}\right)=\dfrac{S_3}{A_3} \cong C_2$ . Since  $C_2$ is nilpotent, $$\dfrac{E_{2}(S_3)}{A_3} = C_2 \implies E_2=S_3.$$
\end{itemize}
Hence, the $E_{3,2}$-series of $S_3$ is
$$
1 = E_{0}(S_3) < E_{1}(S_3)=A_3 < E_{2}(S_3) = S_3.
$$
Thus, $S_3$ is $E_{3,2}$-nilpotent, and its $E_{3,2}$-series has length $3$.

Now, if we consider $m=2$ and $n=3$, we have that $E_1(S_3) = D_{2,3}(S_3) = 1$. Thus, $E_2(S_3) = A_3$ and so $E_3(S_3) = S_3$. Hence, the $E_{2,3}$-series of $S_3$ is 
$$
1 = E_{0}(S_3)= E_{1}(S_3) < E_{2}(S_3) = A_3 < E_3(S_3)=S_3.
$$
Thus, $S_3$ is not $E_{2,3}$-nilpotent.
\end{exam}

As can be seen, the choice of $m$ and $n$ determines whether the series $E_{m,n}$ is nilpotent. Indeed, the series $E_{3,2}$ of $S_3$ is nilpotent and has length $3$,
whereas the series $E_{2,3}$ is not nilpotent, since it contains repeated terms. Thus, the length of an $E_{m,n}$ series is considered only when the series
is strictly increasing and reaches $G$.

In this section, our primary interest is the study of the minimum length among all $E_{m,n}$-nilpotent series of a given finite group, considering all
possible values of $m$ and $n$ for which such a series exists.

\begin{exam}
Consider $G = S_4$, $m = 8$ and $n = 3$. We have that
\begin{itemize}
    \item $E_{0}(S_4) = 1$;
    \item $E_{1}(S_4) = D_{8,3}(S_4) = C_2 \times C_2$;
    \item $\dfrac{E_{2}(S_4)}{C_2 \times C_2} = D_{3,8}\!\left(\dfrac{S_4}{C_2 \times C_2}\right) = D_{3,8}(S_3) = C_3 \implies E_{2}(S_4) = A_4$;
    \item $\dfrac{E_{3}(S_4)}{A_4} = D_{8,3}\!\left(\dfrac{S_4}{A_4}\right) = \dfrac{S_4}{A_4} \cong C_2$. Since $C_2$ is nilpotent, $$\dfrac{E_{3}(S_4)}{A_4} = C_2 \implies E_3=G.$$
\end{itemize}
Hence, the $E_{8,3}$-series of $S_4$ is
$$
1 = E_{0}(S_4)\;<\; E_{1}(S_4)=C_2\times C_2 \;<\; E_{2}(S_4) = A_4\; < \; E_{3}(S_4) = G.
$$
Thus, $S_4$ is $E_{8,3}$-nilpotent, and its $E_{8,3}$-series has length $4$.
\end{exam}

The computation of the $E_{m,n}$-series of $S_4$ exhibits a nontrivial interplay between the two operators $D_{m,n}(\cdot)$ and $D_{n,m}(\cdot)$. In this example, the series is strictly increasing, and its fourth term coincides with the whole group.

The next result shows that the first step of the $E_{m,n}$-series already detects nilpotency: the series collapses at $E_1(G) = G$ if and only if $G$ is nilpotent.

\begin{thm}\label{r=2}
Let $G$ be a finite group of order $mn$ with $m,n>1$ and $\gcd(m,n)=1$. Then $G$ possesses an $E_{m,n}$-nilpotent series of length $2$,
$$
1=E_0(G)<E_1(G)=G,
$$
if and only if $G$ is nilpotent.
\end{thm}

\begin{proof}
$(\Rightarrow)$ Suppose that the series has length $2$. Then there are two possibilities:

\begin{itemize}
    \item[(i)] If $G$ is nilpotent, then by definition $E_1(G)=G$;
    \item[(ii)] If $G$ is not nilpotent, then $E_1(G)=D_{m,n}(G)$.
\end{itemize}

However, by Remark 2.6, we know that $D_{m,n}(G)$ is never equal to $G$. Therefore, the only possible situation is when $G$ is nilpotent.

$(\Leftarrow)$ Conversely, assume that $G$ is nilpotent. It follows from the definition that $E_1(G)=G$. Hence, the series reduces to
$$
1=E_0(G)<E_1(G)=G,
$$
which shows that its length is $2$.
\end{proof}

\begin{rem}\label{3.6}
Observe that if $G$ is a finite group and $\{E_i(G)\}_{i \ge 0}$ is an $E_{m,n}$-series of $G$, then the following holds:  if there exists an integer $t \ge 1$ such that $E_t(G) = G$, it follows that $G$ is solvable. 
\end{rem}

\begin{rem}\label{enmnilpotent}
If a finite group $G$ possesses an $E_{m,n}$-series that reaches $G$, then it cannot contain 3 consecutive repeated terms. Furthermore, by Proposition \ref{factor}(i), if $G$ possesses an $E_{m,n}$-series that reaches $G$ but is not $E_{m,n}$-nilpotent, then $G$ is $E_{n,m}$-nilpotent.
\end{rem}

\begin{rem}
If $G$ is a non-abelian simple group, then $D_{m,n}(G) = D_{n,m}(G) = 1$, 
since the subgroups $D_{m,n}(G)$ and $D_{n,m}(G)$ are characteristic. Hence, all the terms of the $E_{m,n}$-series of $G$ are trivial.
\end{rem}

The stabilization phenomenon described in the next result can be observed in both examples: once the quotient $\tfrac{G}{E_1(G)}$ and the subgroup $E_2(G)$ each carry a Frobenius structure, no further refinement is possible. This structural rigidity underlies the classification of short $E_{m,n}$-series given in Theorem \ref{r}.

 \begin{thm}\label{can}
 Let $G$ be a finite group of order $mn$, where $m,n>1$ and $gcd(m,n)=1$. If $G$ is $E_{m,n}$-nilpotent, then its $E_{m,n}$-series has a length of at most $4$.
\end{thm}

\begin{proof}

Suppose, for a contradiction, that the $E_{m,n}$-series of $G$ has length at least $5$. Then
\[
1=E_0(G)<E_1(G)<E_2(G)<E_3(G)<E_4(G)\leq G.
\]
For simplicity, write $E_i=E_i(G)$ for each $i\geq 0$. Since $G$ is not nilpotent, we have $E_1=D_{m,n}(G)$. Moreover, neither $G/E_1$ nor $G/E_2$ is nilpotent. Indeed, otherwise the definition of the $E_{m,n}$-series would imply $E_2=G$ or $E_3=G$, respectively, contrary to our assumption. Hence,
\[
D_{n,m}\left(\frac{G}{E_1}\right)=\frac{E_2}{E_1}
\qquad\text{and}\qquad
D_{m,n}\left(\frac{G}{E_2}\right)=\frac{E_3}{E_2}.
\]

Since $E_1<E_2$, by Theorem~\ref{frob}, $E_2$ is a Frobenius group with kernel $E_1$. Thus, we may write $E_2=E_1\rtimes H$, where $H$ is a Frobenius complement. In particular, $H\cong E_2/E_1=D_{n,m}(G/E_1)$. Since $D_{n,m}(G/E_1)$ is nilpotent, it follows that $H$ is nilpotent.

By the classical theorem on Frobenius complements (see for instance \cite[Theorem 6.17]{I}), every Sylow subgroup of $H$ of odd order is cyclic, whereas its Sylow $2$-subgroup is either cyclic or generalized quaternion. Since $H$ is nilpotent, it is the direct product of its Sylow subgroups. Therefore, $H\cong C_d\times P$, where $d$ is odd and $P$ is a Sylow $2$-subgroup of $H$, possibly trivial. Consequently, either $H$ is cyclic or $H\cong C_d\times Q_{2^s}$, where $s\geq 3$.

Suppose that $H\cong C_d\times Q_{2^s}$. Since $E_2<E_3$, by Theorem~\ref{frob} and the identities established above, $E_3/E_1$ is a Frobenius group with kernel $E_2/E_1$. On the other hand, since $d$ is odd and $Q_{2^s}$ has a unique involution, $E_2/E_1$ also has a unique involution, say $z$. Hence, $z$ is fixed by every automorphism of $E_2/E_1$. However, since $E_3/E_1$ is a Frobenius group with kernel $E_2/E_1$, every nonidentity element of a Frobenius complement acts fixed-point-freely on $E_2/E_1$, which is a contradiction. Therefore, $H$ is cyclic, and so $E_2/E_1$ is cyclic.

Now, let $K=G/E_1$, $A=E_2/E_1$ and $B=E_3/E_1$. Since $A$ is normal in $K$, the conjugation action of $K$ on $A$ induces a homomorphism from $K$ into $\operatorname{Aut}(A)$ whose kernel is $C_K(A)$. As $A$ is cyclic, then $\operatorname{Aut}(A)$ is abelian and so we obtain that $[K,K]\leq C_K(A)$.

Furthermore, since $B$ is a Frobenius group with kernel $A$ and $A$ is abelian, we have $C_B(A)=A$. Since $B\trianglelefteq K$, it follows that
\[
[K,B]\leq B\cap [K,K]\leq B\cap C_K(A)=C_B(A)=A.
\]
Consequently, $B/A\leq Z(K/A)$.

On the other hand, by the definition of the $E_{m,n}$-series, we have $B/A=D_{m,n}(K/A)\neq 1$. Let $1\neq x\in D_{m,n}(K/A)$. Then $o(x)\mid m$. Since $D_{m,n}(K/A)\neq 1$, it follows from the definition of $D_{m,n}$ that $\gcd(\exp(K/A),n)>1$. Hence, by Cauchy's theorem, there exists a nonidentity element $y\in K/A$ such that $o(y)\mid n$.

Since $x\in B/A\leq Z(K/A)$, we have $xy=yx$. Thus, as $\gcd(m,n)=1$, it follows that
\[
o(xy)=\operatorname{lcm}(o(x),o(y))=o(x)o(y).
\]
Since $x\neq 1$, we obtain $o(xy)\nmid n$, contradicting the definition of $D_{m,n}(K/A)$. Therefore, $D_{m,n}(K/A)$ cannot be both nontrivial and contained in $Z(K/A)$, contrary to what we have established above. This contradiction shows that the $E_{m,n}$-series of $G$ has length at most $4$.\end{proof}

\begin{defn} Let $G$ be a finite solvable group. The Fitting radical of $G$, denoted by $F(G)$, is the largest normal nilpotent subgroup of $G$. We define recursively the Fitting series of $G$ by
$$
F_1(G) := F(G), \quad \text{and for } i > 1, \quad 
F_i(G)/F_{i-1}(G) := F\big(G/F_{i-1}(G)\big).
$$

\noindent The smallest integer $h$ such that $F_h(G) = G$ is called the Fitting height of $G$, and it is denoted by $h_F(G)$.

\end{defn}

Since every step of the $E_{m,n}$-series produces a characteristic subgroup with nilpotent quotient, the construction mirrors the behavior of the Fitting series. The following formalizes this connection by showing that the $E_{m,n}$-series imposes a universal upper bound of four on the Fitting height of any group admitting such a series.

\begin{cor}\label{Fitting}Let $G$ be a finite $E_{m,n}$-nilpotent group. Then the Fitting height of $G$ satisfies $h_F(G)\le 4.$
\end{cor}

\begin{rem}
Due to Corollary \ref{Fitting}, we show that the converse of the claim at Remark \ref{3.6} it is not true. In fact, taking $G = C_2 \wr C_3\wr C_5\wr C_7 \wr C_{11}$, where $\wr$ sets the wreath product, we have that $G$ is solvable and $h_F(G)=5$ (see \cite[Example 2]{height}). Thus, by Remark \ref{enmnilpotent} any $E_{m,n}$-series of $G$ is infinite, for each combination of coprimes $m,n$.
\end{rem}

Now, we can classify all $E_{m,n}$-nilpotent finite groups of length $r$.

\begin{thm}\label{r}
Let $G$ be a finite group which possess a finite $E_{m,n}$-series. Let $r$ be the smallest lenght of an $E_{m,n}$-series of $G$. Then
\begin{enumerate}
   \item[(i)] $r=2$ if and only if $G$ is a  nilpotent group. 
   \item[(ii)] If $r=3$, then there exist subgroups $K, H \leq G$ such that $G = K \rtimes H$, where $K \trianglelefteq G$, $|K|=m$, $|H|=n$, and $\operatorname{Fit}(G) \leq K$.
    \item[(iii)] If $r=4$, then there exist normal subgroups $K,L \trianglelefteq G$ such that $L$ is a Frobenius group with kernel $K$ and $G/L$ is nilpotent.  
\end{enumerate}
\end{thm}

\begin{proof}
\begin{enumerate}
\item [$(i)$] It was already proved in Theorem \ref{r=2}.

\item [$(ii)$] Suppose $r=3$. Set $D=D_{m,n}(G).$ Since the $E_{m,n}$-series of $G$ has length $3$, we have that
$D\neq 1$ and $G/D$ is nilpotent. Once that $D \subset L_m(G)$, we also have that $|D|\mid m.$ Consider $d=|D|$. Since $d\mid m$, there is $a\in\mathbb{N}$ such that $m=da.$ 

\noindent Now, by $a\mid m$ and $\gcd(m,n)=1$, we have that $\gcd(a,n)=1.$ Define $a=p_1^{\alpha_1}\cdots p_s^{\alpha_s}.$ For each $i\in\{1,\ldots,s\}$, let $P_i$ the $p_i$-Sylow subgroup of $G/D$. Since $\gcd(a,n)=1$, it holds that $|P_i|=p_i^{\alpha_i}.$
Let $\overline{K}=P_1\cdots P_s \trianglelefteq G/D$ which has order $m/d$. Consider $D\leq K \trianglelefteq G$ such that $K/D \cong \overline{K}$. We have that $|K|=m.$ Thus, $\gcd(|K|,[G:K])=1$ and by Schur--Zassenhaus Theorem, $K$ has a complement $H$ in $G$. Therefore,
$$
G=KH
\qquad\text{and}\qquad
K\cap H=1.
$$
Hence,  $|H|=n.$ Since $K\trianglelefteq G$, we conclude that $G=K\rtimes H.$ It remains show that $
\operatorname{Fit}(G)\leq K.$ By contradiction, suppose that there exists a prime number $q$ such that
$$
q\mid n
\qquad\text{and}\qquad
q\mid |\operatorname{Fit}(G)|.
$$
Let $Q \trianglelefteq G$ the unique $q$-Sylow of $\mathrm{Fit}(G)$. Once that $|D|\mid m$, $q\mid n$ and $\gcd(m,n)=1$, it holds that $D\cap Q=1$. Because $D,Q\trianglelefteq G$, it is valid that $[D,Q]\leq D\cap Q=1$. Let $1\neq x\in D$ and $1\neq y\in Q$. Since $Q$ is a $q$-group, we have that $o(y)=q^s \mid mn$, for some $s\geq 1$. Once know that $q\mid n$ and $\gcd(m,n)=1$, we have that $q\nmid m$ and, hence, $q^s\mid n$. Thus, $y\in L_n(G)\setminus\{1\}$.

Since $x\in D=D_{m,n}(G)$ and $y\in L_n(G)\setminus\{1\}$, we have that $o(xy)\mid n$. By other hand, $o(x)\mid m$. Moreover, $x$ and $y$ commute. Also, $\gcd(m,n)=1$, and so  $\gcd(o(x),o(y))=1$. Consequently, $o(xy)=o(x)o(y)$. Since $x\neq 1$, it follows that $o(x)o(y)\nmid n$. Thereby, $o(xy)\nmid n$, which is a contradiction. Thus, $\operatorname{Fit}(G)$ is a $\pi(m)$-group.

\noindent Now, consider the canonical projection 
$$
\varphi:G\longrightarrow G/K.
$$
Since $\operatorname{Fit}(G)$ is a $\pi(m)$-group, $\varphi(\operatorname{Fit}(G))$ is also a $\pi(m)$-group. By other hand, $\varphi(\operatorname{Fit}(G))\leq G/K$ and $|G/K|=n$. By Lagrange Theorem, $|\varphi(\operatorname{Fit}(G))|\mid n$. Since $\gcd(m,n)=1$, it holds that $\varphi(\operatorname{Fit}(G))=1$ and hence $\operatorname{Fit}(G)\leq\ker\varphi=K$. Thus,
$$
\operatorname{Fit}(G)\leq K.
$$

\item [$(iii)$] If $r=4$, then $E_3=G$ and so $G/E_2$ is nilpotent. By definition, since $E_2 \neq G$, we have that $E_2/E_1=D_{n,m}(G/E_1)$. Once know that $1 \neq D_{m,n}(G)=E_1 \neq E_2$ we have by Theorem \ref{frob}(ii) that $E_2$ is a Frobenius group with kernel $E_1$.  To conclude the proof, it is enough to take $K:=E_1$ and $L:=E_2$.  
\end{enumerate}  \end{proof}

The $E_{m,n}$-series therefore provides a unified framework that connects element-order conditions, Frobenius kernels, and the classical Fitting hierarchy. Its alternating construction captures the precise layers at which non-nilpotent behavior arises, leading to the structural classification summarized in Theorem \ref{r}.

\begin{rem}\label{remarkmain}
\begin{enumerate}
    \item \textbf{Frobenius groups with nilpotent complements:} 
    Let $G = K \rtimes H$ be a finite Frobenius group with kernel $K$ and complement $H$. By Thompson's theorem, $K$ is nilpotent, and $\gcd(|K|, |H|) = 1$. Setting $m = |K|$ and $n = |H|$, Lemma \ref{2.7}(b) and Theorem \ref{Fro1} yield $E_1(G) = D_{m,n}(G) = K$. If the complement $H$ is nilpotent, then $G/E_1(G) \cong H$ is nilpotent, and the definition immediately gives $E_2(G) = G$. Thus, every solvable Frobenius group with a nilpotent complement is a $E_{m,n}$-nilpotent group with series length equal to $3$.

    \item \textbf{Frobenius groups with non-nilpotent complements:} 
    Suppose that $G=K\rtimes H$ is a Frobenius group with kernel $K$ and non-nilpotent complement $H$. Taking $m=|K|$ and $n=|H|$, as in item $(1)$, we have $D_{m,n}(G)=K$, and hence $E_1(G)=K$. Since $G/E_1(G)\cong H$ is not nilpotent and $\gcd(m,n)=1$, it follows that $E_2(G)/E_1(G)=D_{n,m}(G/E_1(G))=1.$ Moreover, since the quotient remains isomorphic to $H$, and both $D_{m,n}(H)$ and $D_{n,m}(H)$ are trivial, we conclude that $E_i(G)=K$ for every $i\geq1.$ Therefore, the $E_{m,n}$-series stabilizes at the Frobenius kernel $K$ and does not reach $G$. In particular, $G$ is not $E_{m,n}$-nilpotent for this choice of parameters.
\end{enumerate}
\end{rem}

\subsection*{Acknowledgements}
We also thank the referees for their careful reading and suggestions and for pointing out some mistakes in earlier version.


\begin{thebibliography}{99}
\bibitem{Amiri} M. Amiri, I. Kashuba and I. Lima, \emph{On the structure of $LC$-nilpotent groups}. Comm. Algebra, 2026. doi:10.1080/00927872.2026.2704972.

\bibitem{Schmidt} A. Ballester-Bolinches, S.F. Kamornikov and X. Yi, \emph{Finite groups with $\sigma$-subnormal Schmidt subgroups}. Bull. Malays. Math. Sci. Soc. {\bf 45} (2022), 2431--2440.  


\bibitem{Baumslag} B. Baumslag and J. Wiegold, \emph{A sufficient condition for nilpotency in a finite group}, arXiv:1411.2877 [math.GR], 2020.



\bibitem{Frobenius} F.G. Frobenius \textit{ Verallgemeinerung des Sylowschen Satzes}. Berliner Sitz (1895), 981--993.



\bibitem{GOLOD} E.S. Golod. \textit{On nil-algebras and finitely approximable p-groups},  Izv. Math. {\bf 28}, 273--276.

\bibitem{GOLODS} E.S. Golod and I.R. Shafarevich. \textit{On the class field tower}, Izv. Math. {\bf 28}, 261--272.

\bibitem{Tiep}
R.M. Guralnick and P.H. Tiep. \textit{Lifting in Frattini covers and a characterization of finite solvable groups}, J. Reine Angew. Math. {\bf 708} (2015), 49-–72.



\bibitem{height} P. Idziak, P. Kawalek, J. Krzaczkowski and A. Weiss. \textit{Equation satisfiability in solvable groups}, Theory Comput. Syst. {\bf 68} (2024), 740--757. 

\bibitem{I} I. Martin Isaacs,  {\it Finite Group Theory},
Graduate Studies in Mathematics, vol. 92, Am. Math. Soc. (2008).


 
\bibitem{More}A. Moretó and A. Saez, 
\textit{Prime divisors of orders of products}, Proc. Roy. Soc. Edinburgh Sect. A, {\bf 149} (2019), 1153--1162.

\bibitem{quat}H. Mousavia,
{\it Finite groups with quaternion Sylow subgroup}, C. R. Math. {\bf 358} (2020), 1097--1099.



\bibitem{Thom}J.G. Thompson, \emph{Nonsolvable finite groups all of whose local subgroups are solvable}, Bull.
Amer. Math. Soc. {\bf74} (1968), 383–-437.



\end{thebibliography}
\end{document}